\documentclass[12pt, reqno]{amsart}
\usepackage{latexsym,amsmath,amssymb, amsthm}
\usepackage{cases, verbatim, esint, todonotes}
\usepackage[active]{srcltx}
\usepackage[colorlinks=true, %citecolor = blue, 
linkcolor=red]{hyperref}
\usepackage{hypernat}
\usepackage{xcolor}

\ifx\pdfoutput\undefined
  \DeclareGraphicsExtensions{.eps}
\else
  \ifx\pdfoutput\relax
    \DeclareGraphicsExtensions{.eps}
  \else
    \ifnum\pdfoutput>0
      \DeclareGraphicsExtensions{.pdf}
    \else
      \DeclareGraphicsExtensions{.eps}
    \fi
  \fi
\fi

\newtheorem{thm}{Theorem}

\newtheorem{prop}[thm]{Proposition}
\newtheorem{cor}[thm]{Corollary}

\theoremstyle{remark}
\newtheorem{rmk}[thm]{Remark}
\newtheorem{example}[thm]{Example}

\theoremstyle{definition}
\newtheorem{defi}[thm]{Definition}

\numberwithin{thm}{section} 
\numberwithin{equation}{section}

\makeatletter

\newcommand{\Rmnum}[1]{\expandafter\@slowromancap\romannumeral #1@}
\makeatother

\newcommand{\la}{\langle}
\newcommand{\ra}{\rangle}

\newcommand{\vep}{\varepsilon}

\newcommand{\dive}{\operatorname{div}}
\newcommand{\tr}{\operatorname{tr}}

\newcommand{\bpm}{\begin{pmatrix}}
\newcommand{\epm}{\end{pmatrix}}

plus 6pt minus 12pt
\title[Lipschitz and convexity preserving in the Heisenberg group]{\protect{Lipschitz continuity and convexity preserving for solutions of semilinear evolution equations in the Heisenberg group}}

\author[Q. Liu]{Qing Liu}
\address{Qing Liu, Department of Applied Mathematics, Fukuoka University, Fukuoka 814-0180, Japan, {\tt qingliu@fukuoka-u.ac.jp}}
\author[J. J. Manfredi]{Juan J. Manfredi}
\address{Juan J. Manfredi, Department of Mathematics, University of Pittsburgh, Pittsburgh, PA 15260, USA, {\tt manfredi@pitt.edu}}
\author[X. Zhou]{Xiaodan Zhou}
\address{Xiaodan Zhou, Department of Mathematics, University of Pittsburgh, Pittsburgh, PA 15260, USA, {\tt xiz78@pitt.edu}}

\date{\today}

\begin{document}

%\sloppy

%\subjclass[2010]{Primary \ToDo}

%\sloppy

%\dsp

\begin{abstract}
In this paper we study viscosity solutions of semilinear parabolic equations in the Heisenberg group. We show  uniqueness of viscosity solutions with exponential growth at space infinity. We also study Lipschitz and horizontal convexity preserving properties under appropriate assumptions. Counterexamples show that in general such properties that are well-known for semilinear and fully nonlinear parabolic equations in the Euclidean spaces do not hold in the Heisenberg group.
\end{abstract}


\maketitle

\section{Introduction}
This paper is concerned with the uniqueness and the Lipschitz and convexity preserving properties for viscous Hamilton-Jacobi equations on the Heisenberg group $\mathbb{H}$:
\begin{numcases}{}
u_t-\tr (A (\nabla^2_Hu)^\ast)+ f(p, \nabla_H u)=0 &\quad \text{ in $\mathbb{H}\times (0, \infty)$},\label{special eqn}\\
u(\cdot, 0)=u_0 &\quad\text{ in $\mathbb{H}$}, \label{special initial}
\end{numcases}
where $A$ is a given $2\times 2$ symmetric positive-semidefinite matrix and the function $f: \mathbb{H}\times \mathbb{R}^2\to \mathbb{R}$ satisfies certain assumptions to be made explicit later.  Here $\nabla_H u$, $(\nabla_H^2 u)^\ast$ are respectively the horizontal gradient and the horizontal symmetrized Hessian of the unknown function $u$ in space, and $u_0$ is a given locally uniformly continuous function in $\mathbb{H}$.  

Many of our results in this work also hold for more general fully nonlinear degenerate parabolic equations of the type
\begin{equation}\label{fully nonlinear eqn}
u_t+F\left(p, \nabla_H u, (\nabla_H^2 u)^\ast\right)=0 \quad\text{ in $\mathbb{H}\times (0, \infty)$}
\end{equation}
under proper regularity assumptions on $F$. We however focus on \eqref{special eqn} for simplicity of exposition.
\subsection{Uniqueness for unbounded solutions}
Motivated by the uniqueness results in $\mathbb{R}^n$ \cite{CNT, BBBL} , we first give a uniqueness result for unbounded viscosity solutions of \eqref{special eqn}--\eqref{special initial}, which is useful in our later discussion about the Lipschitz and convexity preserving properties. To this end, we need the following Lipschitz continuity of $f$.

\begin{enumerate}
\item[(A1)] There exists $L_1>0$ such that
\begin{equation}\label{hamiltonian1}
|f(p, w_1)-f(p, w_2)|\leq L_1 |w_1-w_2|
\end{equation}
for all $p\in \mathbb{H}$ and $w_1, w_2\in \mathbb{R}^2$. 
\item[(A2)] There exists $L_2(\rho)>0$ depending on $\rho>0$ such that
\begin{equation}\label{hamiltonian2}
|f(p, w)-f(q, w)|\leq L_2(\rho) |p\cdot q^{-1}|_G
\end{equation}
for all $p, q\in \mathbb{H}$ with $|p|, |q|\leq \rho$ and all $w\in \mathbb{R}^2$.
\end{enumerate}
Here $|\cdot |_G$ denotes the Kor\'anyi gauge in $\mathbb{H}$, i.e., 
\[
|p|_G=\left((x_p^2+y_p^2)^2+16z_p^2\right)^{1\over 4}
\]
for all $p=(x_p, y_p, z_p)\in \mathbb{H}$. Note that (A2) is not the usual local Lipschitz continuity in $\mathbb{H}$, since the distance between $p, q\in \mathbb{H}$ defined by $d_R(p, q)=|p\cdot q^{-1}|_G$ is invariant only under right translations and therefore not equivalent to the usual gauge metric give by $d_L(p, q)=|p^{-1}\cdot q|_G$ or the Carnot-Carath\'eodory metric; see Section \ref{sec metric} for more details. 

Our comparison principle is as below. 
\begin{thm}[Comparison principle for unbounded solutions]\label{comparison thm}
Assume that the Lipschitz conditions (A1) and (A2) hold. Let $u$ and $v$ be respectively an upper semicontinuous subsolution and a lower semicontinuous supersolution of \eqref{special eqn}. Assume that for any fixed $T>0$, there exist $k>0$ and $C_T>0$ depending on $T$ such that
\begin{equation}\label{exp growth}
u(p, t)-v(p, t)\leq C_T e^{k\langle p\rangle}
\end{equation}
for all $(p, t)\in \mathbb{H}\times [0, T]$, where  
\begin{equation}\label{bracket}
\la p\ra=(1+x^4+y^4+16z^2)^{1\over 4} \quad \text{ for all $p=(x, y, z)\in \mathbb{H}$.}
\end{equation} 
If $u(p, 0)\leq v(p, 0)$ for all $p\in \mathbb{H}$, then $u\leq v$ in $\mathbb{H}\times [0, \infty)$.
\end{thm}

As an immediate consequence (Corollary \ref{uniqueness thm}), viscosity solutions of \eqref{special eqn} are unique within the class of functions satisfying the following exponential growth condition at infinity:

\begin{enumerate}
\item[(G)] For any $T>0$, there exists $k>0$ and $C_T>0$ such that $|u(p, t)| \leq C_T e^{k\langle p\rangle}$ for all $(p, t)\in \mathbb{H}\times [0, T]$.
\end{enumerate}

Uniqueness of viscosity solutions of various nonlinear equations in the Heisenberg group are studied in \cite{Bi1, MSt, Bi2,  Wa2, Mnotes} etc. It turns out that one may extend the Euclidean viscosity theory (e.g., \cite{CIL}) to sub-Riemannian manifolds. But most of these results are either for a bounded domain or for bounded solutions. It is less understood when the domain and solution are both unbounded in the Heisenberg group. To the best of our knowledge, the only known result on uniqueness for time-dependent equations in this case is due to Haller Martin \cite{HE}, where a comparison principle is established for a class of nonlinear parabolic equations including the horizontal Gauss curvature flow of graphs in the Carnot group. The comparison principle in \cite{HE} is for solutions with polynomial growth at infinity while ours is for exponential growth, but our assumptions on the structure of the equations are stronger.

%The operator $F$ is assumed to be uniformly continuous and satisfy the ellipticity:
%\begin{equation}\label{ellipticity}
%F(p, X)\geq F(p, Y) \quad \text{ for all $p\in \mathbb{R}^2$ and $X, Y\in \mathbf{S}^2$ with $X\leq Y$}.
%\end{equation}
\subsection{Lipschitz and convexity preserving}

In the Euclidean space, Lipschitz continuity and convexity preserving are two very important properties, closely related to the maximum or comparison principle, which hold for a large class of linear and nonlinear parabolic equations: when the initial value $u_0$ is Lipschitz continuous (resp., convex), the unique solution $u(x, t)$ is Lipschitz continuous (resp., convex) in $x$ as well for any $t\geq 0$. Concerning the convexity preserving property in $\mathbb{R}^n$, we refer the reader to \cite{Kor, K2, Sak, GGIS, ALL, Gbook, Ju} for a standard PDE approach in different contexts based on convexity (or concavity) maximum principles and \cite{LSZ} for proofs using the discrete games introduced in \cite{KS1, PSSW, MPR1}. 

In what follows, assuming appropriate growth conditions for the initial value $u_0$ and its derivatives, we sketch a proof of these properties for the unique smooth solution of the classical heat equation:
\begin{equation}
u_t-\Delta u= 0 \quad \text{ in $\mathbb{R}^n\times (0, \infty)$}, \label{nd heat}\\
\end{equation}
with $u(\cdot, 0)=u_0(\cdot )$ in $\mathbb{R}^n$, where $\Delta u$ denotes the usual (Euclidean) Laplacian operator acted on $u$.

By differentiating the equation with respect to the space variables, one may easily see that each of the components of $\nabla u$ satisfies the heat equation \eqref{nd heat}, which, by the maximum principle, implies that $\nabla u(\cdot , t)$ is bounded for any $t\geq 0$ if $\nabla u_0$ is bounded in $\mathbb{R}^n$.

A similar argument works for the convexity preserving property. Indeed, it is not difficult to find that, for any fixed vector $w\in \mathbb{R}^n$, $\la \nabla^2 u w, w\ra$ satisfies the heat equation. One may apply the maximum principle again to show $\la \nabla^2 u(\cdot , t) w, w\ra\geq 0$ for any $t\geq 0$ if it holds initially, which is equivalent to the statement of convexity preserving.

We intend to extend these preserving properties to nonlinear equations in the Heisenberg group $\mathbb{H}$. Notions and properties of Lipschitz continuity and convexity in the Heisenberg group are available in the literature \cite{DGN, LMS, JLMS}. In fact, a function $u$ is said to be Lipschitz continuous in $\mathbb{H}$ if there exists $L>0$ such that
\[
|u(p)-u(q)|\leq Ld_L(p, q)
\]
for all $p, q\in \mathbb{H}$, and $u$ is said to be horizontally convex in $\mathbb{H}$ if
\[
u(p\cdot h^{-1})+u(p\cdot h)\geq 2u(p)
\]
for any $p\in \mathbb{H}$ and any $h\in \mathbb{H}_0$, where 
\[
\mathbb{H}_0=\{h\in \mathbb{H}: h=(h_1, h_2, 0) \text{ for $h_1, h_2\in \mathbb{R}$}\}.
\]
It is clear that  Lipschitz continuity and horizontal convexity are both left invariant.

It is worth stressing that our generalization is by no means immediate. As observed above, besides necessary applications of a comparison principle, the key in the straightforward {proofs} for the Euclidean case lies at differentiating the equation and interchanging derivatives. This is however not applicable directly in the Heisenberg group, since the mixed second derivatives in the Heisenberg group are not commutative in general. In fact, our counterexamples show that preserving of Lipschitz continuity and horizontal convexity may fail even for very simple linear equations; see Examples \ref{ex lip pre} and \ref{linear first example} for the linear equation 
\begin{equation}\label{linear first eqn}
u_t-\langle h_0, \nabla_H u\rangle=0  \quad \text{in $\mathbb{H}$},
\end{equation}
where $h_0\in \mathbb{R}^2$ is given. Its unique viscosity solution turns out to be right translations of the initial value. 

Since the horizontal gradient $\nabla_H u$ and horizontal Hessian $\nabla_H^2 u$ are not  in general right invariant but  only left invariant, we cannot rely on the symmetry of second derivatives for our study of Lipschitz and convexity preserving properties. 

On the other hand, there are many examples on Lipschitz and convexity preserving in the Heisenberg group. One sufficient condition for the equivalence between Lipschitz continuity/horizontal convexity of a function with respect to both metrics $d_L$ and $d_R$ is evenness or vertical evenness of the function; see Definition \ref{def even}, Proposition \ref{prop evenness} and Proposition \ref{evenness lemma}.  Another sufficient condition for the equivalence of both convexity notions is a separable structure of the function (Proposition \ref{separability lemma}).

We thus can obtain the Lipschitz continuity and convexity preserving properties by first investigating them with respect to the right invariant metric $d_R$ and then using the additional assumptions above. Let us present our results in a simpler case.

\begin{thm}[Preserving of right invariant Lipschitz continuity]\label{thm0 lip}
Assume that $f: \mathbb{R}^2\to \mathbb{R}$ is Lipschitz. Let $u\in C(\mathbb{H}\times [0, \infty))$ be the unique solution of 
\begin{equation}\label{simple eqn}
u_t-\tr (A (\nabla^2_Hu)^\ast)+ f(\nabla_H u)=0 \quad \text{ in $\mathbb{H}\times (0, \infty)$},
\end{equation}
with $u(\cdot, 0)=u_0(\cdot)$ satisfying the growth condition (G). If there exists $L>0$ such that
\[
|u_0(p)-u_0(q)|\leq Ld_R(p, q)
\]
for all $p, q\in \mathbb{H}$, then
\[
|u(p, t)-u(q, t)|\leq Ld_R(p, q)
\]
for all $p, q\in \mathbb{H}$ and $t\geq 0$.
\end{thm}
Theorem \ref{thm0 lip} is a direct application of Theorem \ref{comparison thm}. A more general version is given below  in Theorem \ref{thm lip pre}. It implies the Lipschitz preserving property of an even function or vertically even function (Corollary \ref{cor lip}). 

For the case of first order Hamilton-Jacobi equations ($A=0$), if in addition we assume that $f:\mathbb{R}^2\to \mathbb{R}$ is in the form that $f(\xi)=m(|\xi|)$ with $m: \mathbb{R}\to \mathbb{R}$ locally uniformly continuous, then the Lipschitz preserving property of a bounded solution can be directly shown without the evenness assumption. We refer the reader to Theorem \ref{thm lip HJ}, which answers a question asked in \cite{MSt}. A more general question on Lipschitz continuity of viscosity solutions was posed in \cite{BCP}, but it is not clear if our method here immediately applies to that general setting.

As for the h-convexity preserving property, we obtain the following. 
\begin{thm}[Right invariant h-convexity preserving]\label{thm0 left h-convexity}
 Assume that $f: \mathbb{R}^2\to \mathbb{R}$ is Lipschitz. Let $u\in C(\mathbb{H}\times [0, \infty))$ be the unique solution of \eqref{simple eqn} with $u(\cdot, 0)=u_0(\cdot)$ satisfying the growth condition (G). Assume in addition that $f$ is concave in $\mathbb{R}^2$, i.e., 
\begin{equation}\label{assumption operator0}
f(\xi)+f(\eta)\leq 2f\left({1\over 2}(\xi+\eta)\right)
\end{equation}
for all $\xi, \eta\in \mathbb{R}^2$. If $u_0$ is right invariant h-convex in $\mathbb{H}$; that is,
\[
u_0(h^{-1}\cdot p)+u_0(h\cdot p)\geq 2u(p)
\]
for all $p\in \mathbb{H}$ and $h\in \mathbb{H}_0$, then so is $u(\cdot, t)$ for all $t\geq 0$.
\end{thm}

The convexity preserving property for solutions that are either even or in a separable form follows easily (Corollary \ref{cor convexity}). 

Our study of the convexity preserving property in the Heisenberg group is also inspired by recent works on horizontal mean curvature flow in sub-Riemannian manifolds \cite{CC, FLM1}. The mean curvature flow in $\mathbb{R}^n$ is known to preserve convexity \cite{Hui}, but it is not clear if such a property also holds in $\mathbb{H}$ in general.  Our analysis about convexity is only for the simpler equation \eqref{special eqn}. However, an explicit solution of the mean curvature flow in $\mathbb{H}$  that does preserve convexity can be found in Example \ref{example mcf}; see also \cite{FLM1}. 

In the proof of Theorem \ref{thm0 left h-convexity}, we show a convexity maximum principle, following the proof of Theorem \ref{comparison thm}. A general version of this theorem for the equation \eqref{special eqn} is given in Theorem \ref{thm left h-convexity}, where $f: \mathbb{H}\times \mathbb{R}^2$ is assumed to be (right invariant) concave.  One may further generalize this result for \eqref{fully nonlinear eqn} by assuming that $F$ is concave in all arguments. We remark that in the Euclidean case as studied in \cite{GGIS, Ju} etc., there is no need to assume \eqref{assumption operator0}. We need this assumption due to lack of an equivalent definition of horizontal convexity in terms of averages of endpoints. More precisely, convexity of a function $u\in C(\mathbb{R}^n)$ can be expressed by 
\[
u(\xi)+u(\eta)\geq 2u\left({\xi+\eta\over 2}\right)
\] 
for any $\xi, \eta\in \mathbb{R}^n$. However, for horizontal convexity in $\mathbb{H}$ there is no such ``global'' expression valid for all pairs $p, q\in \mathbb{H}$ that are only horizontally related, i.e., $p=q\cdot h$ for $h\in \mathbb{H}_0$. It is not clear to us if this assumption can be dropped in our theorem. 

This paper is organized in the following way. In Section \ref{sec prelim}, we present some basic and useful facts about the Heisenberg group, including an introduction of its metrics, Lipschitz continuity and horizontal convexity. In Section \ref{sec unique}, we give a proof of Theorem \ref{comparison thm} and also include an existence result at the end. The Lipschitz preserving property is studied in Section \ref{sec lip}. Section \ref{sec convex} is dedicated to a discussion of convexity preserving property with several explicit examples in Section \ref{sec example}. 

We thank the referee for a careful review and for helpful comments that improved the readability of the paper.

\section{Preliminaries}\label{sec prelim}
\subsection{ Review of the Heisenberg group $\mathbb{H}$}
Recall that the Heisenberg group $\mathbb{H}$ is $\mathbb{R}^{3}$ endowed with the non-commutative group multiplication 
\[
(x_p, y_p, z_p)\cdot (x_q, y_q, z_q)=\left(x_p+x_q, y_p+y_q, z_p+z_q+\frac{1}{2}(x_py_q-x_qy_p)\right),
\]
for all $p=(x_p, y_p, z_p)$ and $q=(x_q, y_q, z_q)$ in $\mathcal{H}$. Note that the group inverse
of $p=(x_q, y_q, z_q)$ is  $p^{-1}=(-x_q, -y_q, -z_q)$. The Kor\'anyi gauge is given by
\[
|p|_G=((p_1^2+p_2^2)^2+16 p_3^2)^{1/4},
\]
and the left-invariant Kor\'anyi or gauge metric is 
\[
d_L(p, q)=|p^{-1}\cdot q|_G.
\]
The Lie Algebra of $\mathbb{H}$ is generated by the left-invariant vector fields
\[
\begin{aligned}
& X_1=\frac{\partial}{\partial x}-\frac{y}{2}\frac{\partial}{\partial z};\\
& X_2=\frac{\partial}{\partial y}+\frac{x}{2}\frac{\partial}{\partial z};\\
& X_3=\frac{\partial}{\partial z}.
\end{aligned}
\]
  One may easily verify the commuting relation $X_3=[X_1, X_2]=
  X_{1}X_{2}- X_{2}X_{1}$.
  
%Denote by $\mathbf{S}^n$ the set of all symmetric $n\times n$ matrices.
The horizontal gradient  of $u$ is given by 
\[
\nabla_H u=(X_1 u, X_2 u)
\]
and the symmetrized second horizontal Hessian  $(\nabla_H^2 u)^\ast\in S^{2\times 2}$ is given by  
\[
(\nabla_H^2 u)^\ast:=\left(\begin{array}{cc} X_1^2 u & (X_1X_2 u+X_2X_1 u)/2\\
(X_1X_2 u+X_2X_1 u)/2 & X_2^2 u\end{array}\right).
\]
Here $S^{n\times n}$ denotes the set of all $n\times n$ symmetric matrices.
%We conclude our calculations for the second derivatives of $f$ by pointing out that $%
%(\nabla_H^{2, p} f)^\ast=(\nabla_H^{2, q} f)^\ast$. 

A piecewise smooth curve $s\mapsto \gamma(s)\in \mathbb{H}$ is called horizontal if its tangent vector $\gamma'(s)$ is in the linear span of $\{X_1(\gamma(s)), X_2(\gamma(s))\}$ for every $s$ such that $\gamma'(s)$ exists; in other words, there exist $a(s), b(s)\in \mathbb{R}$ satisfying 
\[
\gamma'(s)=a(s)X_1(\gamma(s))+b(s)X_2(\gamma(s))
\]
whenever $\gamma'(s)$ exists.
We denote 
\[
\|\gamma'(s)\|=\left(a^2(s)+b^2(s)\right)^{1\over 2}.
\]
Given $p, q\in \mathbb{H}$, denote
\[
\Gamma(p, q)=\{\text{horizontal curves $\gamma(s)$ ($s\in [0, 1]$): $\gamma(0)=p$ and $\gamma(1)=q$}\}.
\]
Chow's theorem states that $\Gamma(p, q)\neq \emptyset$; see, for example, \cite{BR}. The Carnot-Carath\'eodory metric is then defined to be
\[
d_{CC}(p, q)=\inf_{\gamma\in \Gamma(p, q)}\int_0^1 \|\gamma'(s)\|\, ds.
\]

\subsection{Metrics on $\mathbb{H}$}\label{sec metric}
Besides the left-invariant Kor\'anyi metric $d_L$ and Carnot-Carath\'eodory metric $d_{CC}$, the function $d_R(p, q)=|p\cdot q^{-1}|_G$ for any $p, q\in \mathbb{H}$ defines another metric on $H$, which is right invariant;  in fact, $d_R(p, q)=d_L(p^{-1}, q^{-1})$ for any $p, q\in \mathbb{H}$. 

It is known that $d_L$ is bi-Lipschitz equivalent to the Carnot-Carath\'eodory metric $d_{CC}$ \cite{CDPT, Mnotes}. The metrics $d_L$ and $d_R$ are not bi-Lipschitz equivalent, which is indicated in the example below.

\begin{example}\label{lip ex1}
One may choose 
\[
p=(1-\vep, 1+\vep, {\vep}), \quad q=(1, 1, 0)
\]
with $\vep>0$ small, then by direct calculation, we have $d_L(p, q)^4=|q^{-1}\cdot p|_G^4=4\vep^4$ and $d_R(p, q)^4=|p\cdot q^{-1}|_G^4=4\vep^4+64\vep^2$, which indicates that one cannot expect the existence of a constant $C>0$ such that $d_R(p, q)\leq Cd_L(p, q)$ for all $p, q\in \mathbb{H}$. A variant of this example shows that the reverse inequality also fails in general. 
\end{example}

Although the metrics above are not bi-Lipschitz equivalent, it turns out that one is locally H\"{o}lder continuous in the other. 

\begin{prop}\label{prop holder}
For any $\rho>0$, there exists $C_\rho>0$ such that 
\begin{equation}\label{holder1}
d_L(p, q)\leq C_\rho d_R(p, q)^{1\over 2}
\end{equation}
and
\begin{equation}\label{holder2}
d_R(p, q)\leq C_\rho d_L(p, q)^{1\over 2}
\end{equation}
for any $p, q\in \mathbb{H}$ with $|p|, |q|\leq \rho$.
\end{prop}

\begin{proof}
We give a proof for the sake of completeness. We only show \eqref{holder1}. The proof of \eqref{holder2} is similar. Set $p=(x_p, y_p, z_p)$ and $q=(x_q, y_q, z_q)$.
It is then clear that we only need to show that there exists some $C>0$ depending only on $\rho$ such that
\[
\begin{aligned}
\bigg|z_p-z_q+ {1\over 2} & x_py_q  -{1\over 2}x_qy_p \bigg| \\
&\leq C\left((|x_p-x_q|^2+|y_p-y_q|^2)^2+16\left(z_p-z_q-{1\over 2}x_py_q+{1\over 2}x_qy_p\right)^2\right)^{1\over 4}
\end{aligned}
\]
for all $p, q\in \mathbb{H}$ with $|p|, |q|\leq \rho$. 

Let $\delta=((|x_p-x_q|^2+|y_p-y_q|^2)^2+16|z_p-z_q-{1\over 2}x_py_q+{1\over 2}x_qy_p|^2)^{1\over 4}\leq 1$. Then it is clear that 
\[
\begin{aligned}
\left|z_p-z_q+{1\over 2}x_py_q-{1\over 2}x_qy_p\right|& \leq {\delta^2\over 4}+\left|x_py_q-x_qy_p\right|\\
&={\delta^2\over 4}+\left|(x_p-x_q)y_q-x_q(y_p-y_q)\right|.
\end{aligned}
\]
It follows that 
\[
\left|z_p-z_q+{1\over 2}x_py_q-{1\over 2}x_qy_p\right|\leq {\delta^2\over 4}+(|x_p-x_q|^2+|y_p-y_q|^2)^{1\over 2}(x_q^2+y_q^2)^{1\over 2}.
\]
Noticing that $x_q^2+y_q^2\leq \rho^2$, we have 
\[
\left|z_p-z_q+{1\over 2}x_py_q-{1\over 2}x_qy_p\right|\leq {\delta^2\over 4}+\rho(|x_p-x_q|^2+|y_p-y_q|^2)^{1\over 2}\leq {\delta^2\over 4}+{\rho}\delta=\left({\delta\over 4} +{\rho}\right)\delta.
\]
We conclude the proof by choosing $C=1/4+\rho$. 
\end{proof}

\subsection{Lipschitz continuity}

We discuss two types of Lipschitz continuity with respect to $d_L$ and $d_R$. 

\begin{comment}
\begin{defi}\label{def lip}
A function $u: \mathbb{H}\to \mathbb{R}$ is said to be Lipschitz continuous with respect to $d_L$ if there exists $L>0$ such that
\begin{equation}\label{left lip}
|u(p)-u(q)|\leq Ld_L(p, q)
\end{equation}
for  any $p, q\in \mathbb{H}$. 
\end{defi}

\begin{defi}\label{def right lip}
A function $u: \mathbb{H}\to \mathbb{R}$ is said to be Lipschitz continuous with respect to $d_R$ if there exists $L>0$ such that
\begin{equation}\label{right lip}
|u(p)-u(q)|\leq Ld_R(p, q)
\end{equation}
for  any $p, q\in \mathbb{H}$. 
\end{defi}
We may accordingly define local Lipschitz continuity for both metrics. 
\end{comment}

It is easily seen that the function $f_0:\mathbb{H}\to \mathbb{R}$ given by $f_0(p)=|p|_G$ is a Lipschitz function with respect to $d_L$ and $d_R$, due to the triangle inequality.
%\[
%\begin{aligned}
%&|p|_G-|q|_G=d_L(p, 0)-d_L(q, 0)\leq d_L(p, q);\\
%&|p|_G-|q|_G=d_R(p, 0)-d_R(q, 0)\leq d_R(p, q).
%\end{aligned}
%\]
But there exist functions that are Lipschitz with respect to one of the metrics but not with respect to the other. An example, following Example \ref{lip ex1}, is as below.
\begin{example}\label{lip ex2}
Fix $q=(1, 1, 0)\in \mathbb{H}$ as in Example \ref{lip ex1}. Let $f_q: \mathbb{H}\to \mathbb{R}$ defined by $f_q(p)=d_R(p, q)$ for every $p\in \mathbb{H}$, which satisfies 
\[
|f_q(p)-f_q(p')|=|d_R(p, q)-d_R(p', q)|\leq d_R(p, p')
\]
for all $p, p' \in \mathbb{H}$. But there is no constant $L>0$ such that 
\[
f_q(p)-f_q(p')\leq Ld_L(p, p')
\]
for all $p, p'\in \mathbb{H}$, for otherwise we may take $p=(1-\vep, 1+\vep, \vep)$ and $p'=q$, and get
\[
d_R(p, p')\leq Ld_L(p, p'),
\]
which is not true when $\vep>0$ small, as explained in Example \ref{lip ex1}. However, by Proposition \ref{prop holder}, the function $f_q$ is still locally {1/2}-H\"{o}lder continuous with respect to $d_L$. 
\end{example}

On the other hand, not all functions that are (locally) Lipschitz with respect to $d_L$ or $d_R$ are (locally) Lipschitz with respect to the Euclidean metric. The simplest example is the function $f(p)=|p|_G$ for $p\in \mathbb{H}$.

We conclude this section by showing the equivalence of Lipschitz continuity with respect to both metrics for functions with symmetry. We include in our discussions two different types of evenness.

\begin{defi}[Even functions]\label{def even}
We say a function $f:\mathbb{H}\to \mathbb{R}$ is even (or symmetric about the origin) if   
$f(p)=f(p^{-1})$ for all $p\in \mathbb{H}$. We say $f$ is vertically even (or symmetric about the horizontal coordinate plane) if $f(p)=f(\overline{p})$ for all $p\in \mathbb{H}$, where 
\begin{equation}\label{reflection}
\overline{p}=(x, y, -z) \ \text{ for any $p=(x, y, z)\in \mathbb{H}$}.
\end{equation}
\end{defi}
 
Since $|p\cdot q^{-1}|_G=|\overline{p}^{-1}\cdot \overline{q}|_G=|(p^{-1})^{-1}\cdot q^{-1}|_G$ for any $p, q\in \mathbb{H}$, the following result is obvious.
\begin{prop}[Equivalence of Lipschitz continuities]\label{prop evenness}
Let $f: \mathbb{H}\to \mathbb{R}$ be a function that is either even or vertically even in $\mathbb{H}$.
Then $f$ is Lipschitz continuous with respect to $d_L$ if and only if $f$ is Lipschitz continuous with respect to $d_R$. 
\end{prop}

\subsection{Horizontal convexity}
\begin{defi}[{\cite[Definition 4.1]{LMS}}]\label{defi h-convex}
Let $\Omega$ be an open set in $\mathbb{H}$ and $u: \Omega\to \mathbb{R}$ be an upper semicontinuous function. The function $u$ is said to be horizontally convex or h-convex in $\Omega$, if for every $p\in \mathbb{H}$ and $h\in \mathbb{H}_0$ such that $[p\cdot h^{-1}, p\cdot h]\subset \Omega$, we have 
\begin{equation}\label{h-convex eqn}
u(p\cdot h^{-1})+u(p\cdot h)\geq 2u(p).
\end{equation}
\end{defi}
%An equivalent notion of convexity is called CC-convex; see \cite[Definition 4.2]{LMS}. 

One may also define convexity of a function through its second derivatives in the viscosity sense. 
\begin{defi}\label{defi v-convex}
Let $\Omega$ be an open set in $\mathbb{H}$ and $u:\Omega\to \mathbb{R}$ be an upper semicontinuous function. The function $u$ is said to be v-convex in $\Omega$ if
\begin{equation}\label{v-convex eqn}
(\nabla^2_H u)^\ast(p)\geq 0  \quad \text{for all $p\in \mathbb{H}$}
\end{equation}
in the viscosity sense. 
\end{defi}
It is clear that $u\in C^2(\Omega)$ is v-convex if it satisfies \eqref{v-convex eqn} everywhere in $\Omega$. It is known that the h-convexity and v-convexity are equivalent \cite{LMS, Wa}.  The following example shows that h-convexity is very different from convexity in the Euclidean sense.

\begin{example}\label{ex h-convex}
Let 
\begin{equation}\label{h-convex fun}
f(x, y, z)=x^2y^2+2z^2
\end{equation}
 for all $(x, y, z)\in \mathbb{H}$. It is not difficult to verify that $f$ is h-convex. Indeed, for any $p=(x, y, z)\in \mathbb{H}$ and $h=(h_1, h_2, 0)\in \mathbb{H}_0$, we have
\[
\begin{aligned}
& f(p\cdot h)+f(p\cdot h^{-1})\\
=&\ 2x^2y^2+4z^2+3x^2h_2^2+3y^2h_1^2+2h_1^2h_2^2+6xyh_1h_2\\
\geq &\ 2f(p)+3(xh_2+yh_1)^2+2h_1^2h_2^2\geq 2f(p).
 \end{aligned}
\]
The function $f$ is an example of (globally) h-convex functions in $\mathbb{H}$ that is not convex in $\mathbb{R}^3$.
\end{example}

\section{Uniqueness of unbounded solutions}\label{sec unique}

In this section, motivated by a Euclidean argument in \cite{BBBL}, we present a proof of Theorem \ref{comparison thm} on a comparison principle for \eqref{special eqn} with exponential growth at space infinity. Our result and proof are different from those of \cite{HE}.

\begin{proof}[Proof of Theorem \ref{comparison thm}]
%Let $\mu>0$ be the constant as in \eqref{pen est}. 
We aim to show that $u\leq v$ in $\mathbb{H}\times [0, T)$ for any fixed $T>0$. By the growth assumption, there exist $k>0$ and $C_T>0$ satisfying \eqref{exp growth}. Take an arbitrary constant $\beta>\min\{k, 1\}$ and then $\alpha>0$ to be determined later. 
%depending on $\beta$, $\|A\|$ and $L_f$. %It will actually appear later that $\alpha$ can be any constant satisfying 
%\[
%\alpha>2(\beta^2 \mu^2 +\beta \mu)\|A\|+\beta \mu  L_f.
%\]
Set
\begin{equation}\label{penalty}
g(p, t)=e^{\alpha t+\beta \la p\ra}
\end{equation}
for $(p, t)\in \mathbb{H}\times [0, \infty)$.  {Recall that $\la p \ra$ is a function of $p\in \mathbb{H}$ given in \eqref{bracket}.}
If $p=(x, y, z)$, we have by direct calculations 
\begin{equation}\label{bracket der}
\nabla_H \la p\ra=\left(\frac{ x^3-4y z}{(1+x^4+y^4+16z^2)^{3\over 4}}, \frac{ y^3+4x z}{(1+x^4+y^4+16z^2)^{3\over 4}}\right),
\end{equation}
which implies that there exists $\mu>0$ such that 
\begin{equation}\label{penalty est}
|\nabla_H g(p, t)|\leq \beta \mu g(p, t) 
\end{equation}
for all $(p, t)\in \mathbb{H}\times [0, \infty)$.

We assume by contradiction that $u(p, t)-v(p, t)$ takes a positive value at some $(p, t)\in \mathbb{H}\times (0, \infty)$. Then there exists $\sigma\in (0, 1)$ such that 
\[
u(p, t)-v(p, t)-2 \sigma g(p, t)-{\sigma \over T-t}
\]
attains a positive maximum at $(\hat{p}, \hat{t})\in \mathbb{H}\times [0, T)$.  For all $\vep>0$ small, consider the function 
\[
\Phi(p, q, t, s)= u(p, t)-v(q, s)-\sigma \Psi_\vep(p, q, t, s)-\frac{(t-s)^2}{\vep}-{\sigma\over T-t}
\]
with
\[
\Psi_\vep(p, q, t, s)=\varphi_\vep(p, q)+ K(p, q, t, s),
\]
\smallskip
\[
\varphi_\vep(p, q)={1\over \vep}d_R(p, q)^4={|p\cdot q^{-1}|^4\over \vep}, \quad  K(p, q, t, s)=g(p, t)+g(q, s).
\]
Then $\Phi$ attains a positive maximum at some $(p_\vep, q_\vep, t_\vep, s_\vep)\in \mathbb{H}^2\times [0, T)^2$. In particular, 
\[
\Phi(p_\vep, q_\vep, t_\vep, s_\vep)\geq \Phi(\hat{p}, \hat{p}, \hat{t}, \hat{t}),
\]
which implies that 
\begin{equation}\label{comparison1}
\begin{aligned}
{|p_\vep\cdot q_\vep^{-1}|^4\over \vep}+{(t_\vep-s_\vep)^2\over \vep} \leq   & u(p_\vep, t_\vep)-v(q_\vep, s_\vep)-\sigma g(p_\vep, t_\vep)- \sigma g(q_\vep, s_\vep)-{\sigma\over T-t_\vep}\\ 
&-\left(u(\hat{p}, \hat{t})-v(\hat{p}, \hat{t})-2 \sigma g(\hat{p}, \hat{t})-{\sigma\over T-\hat{t}}\right).
\end{aligned}
\end{equation}
Since, due to \eqref{exp growth}, the terms $u(p_\vep, t_\vep)-v(q_\vep, t_\vep)-\sigma g(p_\vep, t_\vep)-\sigma g(q_\vep, s_\vep)$ are bounded from above uniformly in $\vep$, we have 
\begin{equation}\label{lim0}
d_R(p_\vep, q_\vep)\to 0 \text{ and } t_\vep-s_\vep\to 0 \quad \text{as $\vep\to 0$}.
\end{equation}
We notice that $p_\vep, q_\vep$ are bounded, since otherwise the right hand side of \eqref{comparison1} will tend to $-\infty$. Therefore, by taking a subsequence, still indexed by $\vep$, we have $p_\vep, q_\vep\to \overline{p}\in \mathbb{H}$ and $t_\vep, s_\vep\to \overline{t}\in [0, T)$. It follows that 
\[
\begin{aligned}
& \limsup_{\vep\to 0} u(p_\vep, t_\vep)-v(q_\vep, s_\vep)-\sigma g(p_\vep, t_\vep)-\sigma g(q_\vep, s_\vep)-{\sigma\over T-t_\vep}\\
&\leq u(\overline{p}, \overline{t})-v(\overline{p}, \overline{t})-2\sigma g(\overline{p},\overline{t})-{\sigma \over T-\overline{t}},
\end{aligned}
\]
which yields
\[
\varphi_\vep(p_\vep, q_\vep)\to 0 \quad \text{ as $\vep\to 0$}.
\]
Also, it is easily seen that $\overline{t}> 0$ and therefore $t_\vep, s_\vep>0$ thanks to the condition that $u(\cdot, 0)\leq v(\cdot, 0)$ in $\mathbb{H}$.

{In order to apply the Crandall-Ishii lemma (cf. \cite{CIL}) in our current case, let us recall the definition of semijets adapted to the Heisenberg group: for any $(p, t)\in \mathbb{H}\times (0, \infty)$ and any locally bounded upper semicontinuous function $u$ in $\mathbb{H}\times (0, \infty)$, 
\[
\begin{aligned}
P_H^{2, +}u(p, t)=\bigg\{(\tau, \zeta, X)\in \mathbb{R}\times \mathbb{R}^3\times S^{2\times 2}: u(q, s)\leq& u(p, t) +\tau (s-t)\\
+\la \zeta, p^{-1}\cdot q\ra + {1\over 2} & \la X h, h\ra+o(|p^{-1}\cdot q|_G^2)\bigg\},
\end{aligned}
\]
where $h$ denotes the horizontal projection of $p^{-1}\cdot q$. Similarly, we may define 
\[
\begin{aligned}
P_H^{2, -}u(p, t)=\bigg\{(\tau, \zeta, X)\in \mathbb{R}\times \mathbb{R}^3\times S^{2\times 2}: u(q, s)\geq& \ u(p, t) +\tau (s-t)\\
+\la \zeta, p^{-1}\cdot q\ra +{1\over 2} & \la X h, h\ra+o(|p^{-1}\cdot q|_G^2)\bigg\}
\end{aligned}
\]
for any locally bounded lower semicontinuous function $u$. Also, the closure $\overline{P}_H^{2, +}$ is the set of  triples $(\tau, \zeta, X)\in \mathbb{R}\times \mathbb{R}^3\times S^{2\times 2}$ that satisfy the following: there exist $(p_j, t_j)\in \mathbb{H}\times [0, \infty)$ and $(\tau_j, \zeta_j, X_j)\in P_H^{2, +}(p_j, t_j)$ such that 
\[
\left(p_j, t_j, u(p_j, t_j), \tau_j, \zeta_j, X_j\right)\to \left(p, t, u(p, t), \tau, \zeta, X\right) \quad\text{ as $j\to \infty$}.
\]  
The closure set $\overline{P}_H^{2, -}$ of $P_H^{2, -}$ can be similarly defined. We refer to \cite{Bi2} for more details. }

We now apply the adaptation of the Crandall-Ishii lemma to the Heisenberg group \cite{Mnotes, Bi2} and get for any $\lambda\in(0, 1)$
\[
(a_1, \zeta_1, X)\in \overline{P}_H^{2, +} u(p_\vep, t_\vep) \text{  and } (a_2, \zeta_2, Y)\in \overline{P}_H^{2, -} v(q_\vep, s_\vep)
\]
such that
\begin{equation}\label{time der}
 a_1-a_2=\alpha \sigma K(p_\vep, q_\vep, t_\vep, s_\vep)+{\sigma\over (T-t_\vep)^2},
 \end{equation}
 \smallskip
 \begin{equation} \label{ishii}
\langle X w, w\rangle -\langle Y w, w\rangle\leq \la (\sigma M+\lambda \sigma^2 M^2) w_{p_\vep}\oplus w_{q_\vep}, w_{p_\vep}\oplus w_{q_\vep}\ra,
\end{equation}
and the horizontal projections of $\zeta_1, \zeta_2\in \mathbb{R}^3$ can be written respectively as 
$\xi+\eta_1$ and $\xi+\eta_2$ (in $\mathbb{R}^2$) with
 \[
 \xi=\nabla^p_H \varphi_\vep(p_\vep, q_\vep)=-\nabla^q_H \varphi_\vep(p_\vep, q_\vep),
 \]
 \[
\eta_1=\beta \sigma \nabla_H g(p_\vep, t_\vep), \quad \eta_2=-\beta \sigma \nabla_H g(q_\vep, s_\vep).
\]

\smallskip
\noindent Here $w=(w_1, w_2)\in \mathbb{R}^2$ is arbitrary, $M=(\nabla^2  \Psi_\vep)^\ast(p_\vep, q_\vep, t_\vep, s_\vep)$ is a $6\times 6$ symmetric matrix, and\begin{equation}\label{wpe}
w_{p_\vep}=\left(w_1, w_2, {1\over 2}w_2 x_{p_\vep}-{1\over 2}w_1 y_{p_\vep}\right)
\end{equation}
and
\begin{equation}\label{wqe}
w_{q_\vep}=\left(w_1, w_2, {1\over 2}w_2 x_{p_\vep}-{1\over 2}w_1 y_{q_\vep}\right)
\end{equation}
with $p_\vep=(x_{p_\vep}, y_{p_\vep}, z_{p_\vep})$ and $q_\vep=(x_{q_\vep}, y_{q_\vep}, z_{q_\vep})$.

%%By Young's inequality, it is easily seen that 
%%\begin{equation}\label{est1}
%%|\eta|\leq C_1\beta e^{\alpha t_\vep+\beta\la p_\vep\ra}, \quad |\tilde{\eta}|\leq C_1\beta e^{\alpha t_\vep+\beta\la q_\vep\ra}
%%\end{equation}
%%for some constant $C_1>0$.
It is easily seen that $M=M_1+M_2$, where 
\[
M_1=\nabla^2 \varphi_\vep(p_\vep, q_\vep)
\]
and
\[
M_2=\nabla^2 K (p_\vep, q_\vep)=
\begin{pmatrix}{}
\nabla^2 g (p_\vep, t_\vep) & 0 \\ 0 & \nabla^2 g(q_\vep, s_\vep)
\end{pmatrix}.
\]
It follows from the calculation in the comparison arguments in \cite{Bi2} (and also \cite{Bi1, Mnotes}) that there exists $C>0$ such that 
\begin{equation}\label{bieske}
\la (M_1+\lambda M_1^2) w_{p_\vep}\oplus w_{q_\vep}, w_{p_\vep}\oplus w_{q_\vep}\ra\leq {C\over \vep}|w|^2(z_{p_\vep}-z_{q_\vep}-{1\over 2}x_{p_\vep} y_{q_\vep}+{1\over 2}x_{q_\vep} y_{p_\vep})^2
\end{equation}
for any $\lambda>0$ small.  We next follow the strategy in the Euclidean case from \cite{BBBL}. However{,} the algebraic complexity is quite more challenging in the non-commutative case. With the help of a computer algebra system\footnote{Program is available  in the arXiv.org  version of the paper.}, we simplify the left hand side of the following inequalities and obtain  a constant $C_\beta>0$ depending only on $\beta$, such that 
\begin{equation}\label{heisenberg second0}
\la M_2 (w_{p_\vep}\oplus w_{q_\vep}), (w_{p_\vep}\oplus w_{q_\vep})\ra \\
\leq {1 \over \vep} |w|^2C_\beta K(p_\vep, q_\vep, t_\vep, s_\vep) ,
\end{equation}

\begin{equation}\label{heisenberg second1}
\begin{aligned}&\la M_1M_2 (w_{p_\vep}\oplus w_{q_\vep}), (w_{p_\vep}\oplus w_{q_\vep})\ra \\
&\leq {1 \over \vep} |w|^2C_\beta K(p_\vep, q_\vep, t_\vep, s_\vep) \left|z_{p_\vep}-z_{q_\vep}-{1\over 2}x_{p_\vep} y_{q_\vep}+{1\over 2}x_{q_\vep} y_{p_\vep}\right|
\end{aligned}
\end{equation}
and
\begin{equation}\label{heisenberg second2}
\la M_2^2 (w_{p_\vep}\oplus w_{q_\vep}), (w_{p_\vep}\oplus w_{q_\vep}) \ra\leq |w|^2C_\beta K^2(p_\vep, q_\vep, t_\vep, s_\vep).
\end{equation}

We remark that the existence of $C_\beta$ here is essentially due to the boundedness of $\nabla_H \la p \ra$ and $\nabla^2_H \la p \ra$ in $\mathbb{H}$.

By \eqref{ishii} and \eqref{bieske}, we may take $\lambda>0$ sufficiently small, depending on the size of $\vep$, $\overline{p}, \overline{t}$, and $\beta$, such that  
\begin{equation}\label{est2}
\begin{aligned}
&\la Xw, w\ra-\la Yw, w\ra \\
&\leq {C\sigma \over \vep}|w|^2(z_{p_\vep}-z_{q_\vep}-{1\over 2}x_{p_\vep} y_{q_\vep}+{1\over 2}x_{q_\vep} y_{p_\vep})^2+2\sigma |w|^2 C_\beta K(p_\vep, q_\vep, t_\vep, s_\vep).
\end{aligned}
\end{equation}

We next apply the definition of viscosity sub- and supersolutions and get
\begin{equation}\label{vis1}
a_1-\tr (AX)+f(p_\vep, \xi+\eta_1)\leq 0
\end{equation}
and
\begin{equation}\label{vis2}
a_2-\tr (AY)+f(q_\vep, \xi+\eta_2)\geq 0.
\end{equation}
By subtracting \eqref{vis2} from \eqref{vis1}, we have
\[
a_1-a_2\leq \tr (A X)-\tr (A Y)+f(q_\vep, \xi+\eta_2)-f(p_\vep, \xi+\eta_1),
\]
which yields, by \eqref{est2} and (A1), 
\begin{equation}\label{est3}
\begin{aligned}
a_1-a_2\leq {C\sigma \over \vep}\|A\|&(z_{p_\vep}-z_{q_\vep}-{1\over 2}x_{p_\vep} y_{q_\vep}+{1\over 2}x_{q_\vep} y_{p_\vep})^2+L_2(\rho)|p_\vep\cdot q_\vep^{-1}|\\
&+(2\sigma C_\beta\|A\|+2\beta \mu \sigma L_1)K(p_\vep, q_\vep, t_\vep, s_\vep)
\end{aligned}
\end{equation}
with $\rho=|\overline{p}|+1$ for $\vep>0$ sufficiently small.

Since we have \eqref{lim0}, we now can take $\vep>0$ small to get
\[
{C\over \vep}(z_{p_\vep}-z_{q_\vep}-{1\over 2}x_{p_\vep} y_{q_\vep}+{1\over 2}x_{q_\vep} y_{p_\vep})^2+ L_2(\rho)|p_\vep\cdot q_\vep^{-1}|\leq {\sigma\over T^2}.
\]
Taking $\lambda>0$ accordingly small and  
\[
\alpha>2C_\beta\|A\|+\beta \mu  L_f,
\] 
we reach a contradiction to \eqref{time der}.

%%%%%%%%%%%%%%%%%%%%%%%%%%%%
\begin{comment}
{Notes: without symmetrization, the original entries are }
\[
\begin{aligned}
M_\beta^{12}=\beta  & e^{\alpha t_\vep +\beta\la p_\vep\ra} \Bigg(\frac{\beta(x_{r_\vep}-4y_\vep z_\vep)(y_{r_\vep}+4x_\vep z_\vep)}{(1+x_\vep^4+y_\vep^4+16z_{q_\vep})^{3\over 2}}\\
&+\frac{4z_\vep-8z(x_\vep^4-2y_\vep^4)+64z_{r_\vep}-2x_\vep y_\vep-2x_\vep^5y_\vep-2x_\vep y_\vep^5-3x_{r_\vep}y_{r_\vep}+16x_\vep y_\vep z_{q_\vep}}{(1+x_\vep^4+y_\vep^4+16z_{q_\vep})^{7\over 4}}\Bigg);\\
M_\beta^{21}=\beta  & e^{\alpha t_\vep +\beta\la p_\vep\ra} \Bigg(\frac{\beta(x_{r_\vep}-4y_\vep z_\vep)(y_{r_\vep}+4x_\vep z_\vep)}{(1+x_\vep^4+y_\vep^4+16z_{q_\vep})^{3\over 2}}\\
&+\frac{-4z_\vep-8z(2x_\vep^4-y_\vep^4)-64z_{r_\vep}-2x_\vep y_\vep-2x_\vep^5y_\vep-2x_\vep y_\vep^5-3x_{r_\vep}y_{r_\vep}+16x_\vep y_\vep z_{q_\vep}}{(1+x_\vep^4+y_\vep^4+16z_{q_\vep})^{7\over 4}}\Bigg);\\
\end{aligned}
\]
\end{comment}
\end{proof}

An immediate consequence is certainly the uniqueness of solutions with at most exponential growth at space infinity. 

\begin{cor}[Uniqueness of solutions]\label{uniqueness thm}
Assume that (A1) and (A2) hold. Let $u_0\in C(\mathbb{H})$. Then there is at most one continuous viscosity solution $u$ of \eqref{special eqn}--\eqref{special initial} satisfying the exponential growth condition (G).
\end{cor}

The existence of viscosity solutions of \eqref{special eqn}--\eqref{special initial} is not the main topic of this work, but we remark that it is possible to adapt Perron's method \cite{CIL} to our current case in the Heisenberg group, under various extra assumptions on the function $f$. For example, one may further assume on \eqref{special eqn} that
\begin{enumerate}
\item[(A3)] $|f(p, \xi)|\leq C_f(1+|\xi|)$
for some $C_f>0$ and all $p\in \mathbb{H}, \xi\in \mathbb{R}^2$.
\end{enumerate}
In this case, it is not difficult to verify by computation that $\overline{u}=Cg(p, t)+C_f t$ and $\underline{u}=-Cg(p, t)-C_f t$ are respectively a supersolution and a subsolution of \eqref{special eqn} for any $C>0$ and $\beta>0$ when $\alpha>0$ is sufficiently large. Indeed, we have
\[
\overline{u}_t=C\alpha g+C_f,
\]
\[
|\tr(A(\nabla_H^2 \overline{u})^\ast)|\leq C\|A\| \beta^2\mu^2 g,
\]
and
\[
|f(p, \nabla_H \overline{u})|\leq C C_f\beta\mu g+C_f,
\]
where $\mu$ is the same constant as in the proof of Theorem \ref{comparison thm}. Therefore, by (A3), we get 
\[
\overline{u}_t-\tr(A(\nabla_H^2 \overline{u})^\ast)+f(p, \nabla_H \overline{u})\geq 0
\]
when $\alpha> \|A\|\beta^2\mu^2+C_f\beta\mu$. The verification for $\underline{u}$ is similar. 

If there exist $C>0$ and $k>0$ such that 
\begin{equation}\label{initial growth}
-Ce^{k \la p \ra} \leq u_0(p) \leq Ce^{k \la p \ra} \quad \text{ for all $p\in \mathbb{H}$},
\end{equation} 
then classical arguments \cite{CIL} show that the supremum over all subsolutions bounded by $\underline{u}$ and $\overline{u}$ is in fact a unique continuous solution. We state the result below without more details in its proof. 
  
\begin{cor}
Assume the Lipschitz conditions (A1), (A2) and the growth condition (A3). Let $u_0\in C(\mathbb{H})$ satisfy \eqref{initial growth} for some $C>0$ and $k>0$. Then there exists a unique continuous solution $u$ of \eqref{special eqn}--\eqref{special initial} satisfying the exponential growth condition (G).
\end{cor}

\section{Lipschitz preserving properties}\label{sec lip}

In this section, we strengthen the assumption (A2) on $f$; we assume 
\begin{enumerate}
\item[(A2')] the function $f(p, \xi)$ is globally Lipschitz continuous in $p$ with respect to the metric $d_R$, i.e., there exists $L_0>0$ such that
\begin{equation}\label{hamiltonian3}
|f(p, \xi)-f(q, \xi)|\leq L_0 |p\cdot q^{-1}|_G
\end{equation}
for all $p, q\in \mathbb{H}$ and $\xi\in \mathbb{R}^2$.
\end{enumerate}

\subsection{Right invariant Lipschitz continuity preserving}

We first discuss the Lipschitz continuity based on the standard gauge metric $d_L$ (or equivalently, the Carnot-Carath\'eodory metric). It turns out that even the simplest first order linear equation will not preserve such Lipschitz continuity.
\begin{example}\label{ex lip pre}
Fix $h_0=(1, 1)\in \mathbb{R}^2$.  Let us consider the equation 
\[
u_t-\langle h_0, \nabla_H u\rangle=0  \quad \text{in $\mathbb{H}$}
\]
with $u(p, 0)=u_0(p)=|p|_G$ for $p\in \mathbb{H}$. By direct verification and Corollary \ref{uniqueness thm}, the unique solution is 
\[
u(p, t)=|p\cdot ht|_G=d_R(p, h^{-1}t),
\]
where $h=(1, 1, 0)\in \mathbb{H}_0$.
However, it is not Lipschitz continuous with respect to $d_L$. Indeed, similar to Example \ref{lip ex2}, one may choose $p_1=(-t-\vep, -t+\vep, -\vep t)$ and $p_2=tv^{-1}=(-t, -t, 0)$, which gives
\[
u(p_1, t)-u(p_2, t)=|p_1\cdot p_2^{-1}|_G=(4\vep^4+64\vep^2 t^2)^{1\over 4}
\]
but
\[
d_L(p_1, p_2)=|p_1^{-1}\cdot p_2|_G=\sqrt{2}\vep.
\]
\end{example}

The example above directs us to first consider the Lipschitz continuity with respect to $d_R$. The following result is an immediate consequence of Theorem \ref{comparison thm}.
\begin{thm}[Preserving of right invariant Lipschitz continuity]\label{thm lip pre}
Assume that $f: \mathbb{H}\times \mathbb{R}^2\to \mathbb{R}$ satisfies the assumptions (A1), (A2') {and (A3)}. Let $u\in C(\mathbb{H}\times [0, \infty))$ be the unique solution of \eqref{special eqn}--\eqref{special initial}
satisfying the growth condition (G). If there exists $L>0$ such that
\begin{equation}\label{initial lip}
|u_0(p)-u_0(q)|\leq Ld_R(p, q)
\end{equation}
for all $p, q\in \mathbb{H}$, then
\begin{equation}\label{lip preserve}
|u(p, t)-u(q, t)|\leq (L+L_0t)d_R(p, q)
\end{equation}
for all $p, q\in \mathbb{H}$ and $t\geq 0$. In particular, there exists $C_\rho>0$ depending on $\rho>0$ and $t\geq 0$ such that
\begin{equation}\label{holder preserve}
|u(p, t)-u(q, t)|\leq C_\rho d_L(p, q)^{1\over 2}
\end{equation}
for all $p, q\in \mathbb{H}$ with $|p|, |q|\leq \rho$. Moreover, when $f$ does not depend on the space variable $p$,  \eqref{lip preserve} holds with $L_0=0$.
\end{thm}

\begin{proof}
By symmetry, we only need to prove that 
\begin{equation}\label{lip preserve1}
u(p, t)-u(h^{-1}\cdot p, t)\leq (L+L_0t)|h|_G
\end{equation}
for all $p, h\in \mathbb{H}$ and $t\geq 0$. It suffices to show that 
\[
v(p, t)=u(h^{-1}\cdot p, t)+(L+L_0t)|h|_G
\]
is a supersolution of \eqref{special eqn}--\eqref{special initial} for any $h\in \mathbb{H}$. To this end, we recall the left invariance of horizontal derivatives in the Heisenberg group, which implies that $v$ is a supersolution of 
\[
v_t-\tr (A \nabla^2_H v)+ f(h^{-1}\cdot p, \nabla_H v)=L_0|h|_G \quad \text{ in $\mathbb{H}\times (0, \infty)$}.
\]
Since 
\[
|f(h^{-1}\cdot p, \nabla_H v)-f(p, \nabla_H v)|\leq L_0 |h|_G
\]
due to \eqref{hamiltonian3}, we easily see that $v$ is a supersolution of \eqref{special eqn}. Also, by \eqref{initial lip}, we have $u(p, 0)\leq v(p, 0)$ for all $p\in \mathbb{H}$. We conclude the proof of \eqref{lip preserve1} by applying Theorem \ref{comparison thm}.  The H\"{o}lder continuity \eqref{holder preserve} follows from Proposition \ref{prop holder}.
\end{proof}

In view of Proposition \ref{prop evenness}, we may use the theorem above to show the preserving of Lipschitz continuity in the standard gauge metric under the assumption of evenness or vertical evenness. 
\begin{cor}[Lipschitz preserving of even solutions]\label{cor lip}
Assume that $f: \mathbb{H}\times \mathbb{R}^2\to \mathbb{R}$ satisfies the conditions (A1), (A2') {and (A3)}. Let $u\in C(\mathbb{H}\times [0, \infty))$ be the unique solution of \eqref{special eqn}--\eqref{special initial}
satisfying the growth condition (G). Assume also that $u(\cdot, t)$ is an even or vertically even function. If there exists $L>0$ such that
\begin{equation}\label{initial lip2}
|u_0(p)-u_0(q)|\leq Ld_L(p, q)
\end{equation}
for all $p, q\in \mathbb{H}$, then
\begin{equation}\label{lip preserve2}
|u(p, t)-u(q, t)|\leq (L+L_0t)d_L(p, q)
\end{equation}
for all $p, q\in \mathbb{H}$ and $t\geq 0$. In particular, when $f$ does not depend on the space variable $p$, then \eqref{lip preserve2} holds with $L_0=0$.
\end{cor}

\subsection{A special class of Hamilton-Jacobi equations} We discuss the Lipschitz preserving property for bounded solutions of a special class of first order Hamilton-Jacobi equations whose Hamiltonians depend only on the norm of horizontal gradient. More precisely, we study equations in the form of 
\begin{equation}\label{special hamiltonian}
u_t+m(|\nabla_H u|)=0 \quad \text{ in $\mathbb{H}\times (0, \infty)$,}
\end{equation}
where $m:\mathbb{R}\to \mathbb{R}$ is a locally uniformly continuous function, with initial condition $u(\cdot , 0)=u_0(\cdot )$ bounded Lipschitz continuous with respect to $d_L$ in $\mathbb{H}$. Since the assumption on $m$ is quite weak, our uniqueness and existence results for unbounded solutions in Section \ref{sec unique} do not apply.  

For solutions bounded in space, see \cite{MSt} for a uniqueness theorem and a Hopf-Lax formula when the Hamiltonian $\xi\mapsto m(|\xi|)$ is assumed to be convex and coercive. For instance, when $m(|\xi|)=|\xi|^2/2$, the unique solution of \eqref{special hamiltonian} can be expressed as 
\begin{equation}\label{hopf-lax}
u(p, t)=\inf_{q\in \mathbb{H}}\left\{{t\over 2}d_{CC}^2\left(0, \left(q^{-1}\cdot p\over t\right)\right)+u_0(q)\right\}.
\end{equation}

The Lipschitz preserving property (with respect to $d_L$ or $d_{CC}$) was left as an open question in \cite{MSt}; see also \cite{BCP} for a related open question but for more general Hamiltonians. In contrast to the Euclidean case, it is not obvious how to prove the Lipschitz continuity by using the Hopf-Lax formula \eqref{hopf-lax}. We here give an answer to this question using a PDE approach. 

\begin{thm}[Lipschitz preserving for special Hamilton-Jacobi equations]\label{thm lip HJ}
Suppose that $m: \mathbb{R}\to \mathbb{R}$ is locally uniformly continuous.  Let $u$ be the unique viscosity solution of \eqref{special hamiltonian} with $u(\cdot, 0)=u_0(\cdot )$ bounded in $\mathbb{H}$. If $u_0$ is Lipchitz with respect to $d_L$ in $\mathbb{H}$, i.e., there exists $L>0$ such that 
\eqref{initial lip2} holds for any $p, q\in \mathbb{H}$, then for all $t\geq 0$
\[
|u(p, t)-u(q, t)|\leq Ld_L(p, q)
\]
for all $p, q\in \mathbb{H}$.
\end{thm}

\begin{proof}
Under the assumptions above, it is known \cite{MSt} that for any fixed $T>0$, there is a unique bounded continuous viscosity solution in $\mathbb{H}\times [0, T)$. We only need to show that 
\[
u(p, t)-u(q, t)\leq Ld_L(p, q)
\]
for all $p, q\in \mathbb{H}$ and $t\in [0, T)$. The other part can be shown by a symmetric argument.

By Young's inequality applied to \eqref{initial lip2}, we obtain
\begin{equation}\label{initial lip3}
u_0(p)-u_0(q)\leq {Ld_L(p, q)^4\over 4\delta^4}+{3L\delta^{4\over 3}\over 4}
\end{equation}
for all $\delta>0$ and $p, q\in \mathbb{H}$. It then suffices to show that
\begin{equation}\label{goal special Hamiltonian}
u(p, t)-u(q, t)\leq {Ld_L(p, q)^4\over 4\delta^4}+{3L\delta^{4\over 3}\over 4}
\end{equation}
for all $\delta>0$ and $p, q\in \mathbb{H}$. To this end, we fix $\delta>0$ and prove below that 
\[
u_L(p, t)=\inf_{q\in \mathbb{H}}\left\{u(q, t)+{Cd_L(p, q)^4}\right\}
\]
with $C=L/4\delta^4$ is a supersolution of \eqref{special hamiltonian}. Suppose there exist a bounded open set $\mathcal{O}\subset \mathbb{H}\times (0, T)$, $\phi\in C^2(\mathcal{O})$ and $(\hat{p}, \hat{t})\in \mathcal{O}$ such that
\[
(u_L-\phi)(\hat{p}, \hat{t})<(u_L-\phi)(p, t)
\]
for all $(p, t)\in \mathcal{O}$. We may also assume that $\phi(p, t)\to -\infty$ when $(p, t)\to \partial \mathcal{O}$. Then for any $\vep>0$ sufficiently small, 
\[
\Phi_\vep(p, q, t, s)=u(q, t)+{Cd_L(p, q)^4}-\phi(p, s)+{(t-s)^2\over \vep}
\]
attains a minimum at $(p_\vep, q_\vep, t_\vep, s_\vep)\in \mathbb{H}\times \mathbb{H}\times [0, \infty)\times [0, \infty)$. A standard argument yields $p_\vep, q_\vep\to \hat{p}$ and $t_\vep, s_\vep\to \hat{t}$ as $\vep\to 0$, which, in particular, implies that $t_\vep, s_\vep\neq 0$. The minimum also implies that 
\begin{equation}\label{derivative relation}
\nabla_H\phi_1(p_\vep)=\nabla_H \phi(p_\vep, s_\vep) \text{ and } \phi_t(p_\vep, s_\vep)={2(t_\vep-s_\vep)\over \vep},
\end{equation}
where $\phi_1(p)=Cd_L(p, q_\vep)^4$.

We next apply the definition of supersolutions and get
\begin{equation}\label{special v-inequality}
a+m(|\nabla_H\phi_2(q_\vep)|)\geq 0,
\end{equation}
where 
\[
a={2(t_\vep-s_\vep)\over \vep} \ \text{ and } \phi_2(q)=-Cd_L(p_\vep, q)^4.
\]
By \eqref{derivative relation}, in order to prove that $u_L$ is a supersolution, we only need to substitute $\nabla_H\phi_2(q_\vep)$ in \eqref{special v-inequality} with $\nabla_H\phi_1(p_\vep)$.  By direct calculation, we have
\[
\nabla_H^p d_L(p, q)^4=4 \bigg(\delta_1(\delta_1^2+\delta_2^2)-4\delta_2\delta_3,\ \delta_2(\delta_1^2+\delta_2^2)+4\delta_1\delta_3\bigg)
\]
and
\[
\nabla_H^q d_L(p, q)^4=4\bigg(-\delta_1(\delta_1^2+\delta_2^2)-4\delta_2\delta_3,\ -\delta_2(\delta_1^2+\delta_2^2)+4\delta_1\delta_3\bigg)
\]
with $p=(x_p, y_p, z_p)$, $q=(x_q, y_q, z_q)$ and
\[
\delta_1=x_p-x_q,\ \delta_2=y_p-y_q,\ \delta_3=z_p-z_q+{1\over 2}x_py_q-{1\over 2}x_q y_p,
\]
This reveals that $\nabla_H^p d_L(p, q)^4\neq -\nabla_H^q d_L(p, q)^4$ and therefore $\nabla_H\phi_1(p_\vep)\neq \nabla_H\phi_2(q_\vep)$ in general; {see \cite{FLM1} for more details on this aspect}. However, their norms stay the same, i.e., $|\nabla_H^p d_L(p, q)^4|=|\nabla_H^q d_L(p, q)^4|$, which turns out to be a key ingredient in this proof. In fact, we have 
\[
|\nabla_H^p d_L(p, q)^4|=|\nabla_H^q d_L(p, q)^4|=4 d_L(p, q)^2(\delta_1^2+\delta_2^2)^{1\over 2},
\]
which implies that $|\nabla_H\phi_1(p_\vep)|= |\nabla_H\phi_2(q_\vep)|$ and their boundedness uniformly in $\vep$. Hence, due to \eqref{derivative relation}, the equation \eqref{special v-inequality} is now rewritten as 
\[
\phi_t(p_\vep, s_\vep)+m(|\nabla_H \phi(p_\vep, s_\vep)|)\geq 0.
\]
By sending $\vep\to 0$ and using the continuity of $m$, we conclude the verification that $u_L$ is a supersolution. It follows that $v_L=u_L+3L\delta^{4\over 3}/4$ is also a supersolution of \eqref{special hamiltonian}. Thanks to \eqref{initial lip3}, we have $u(p, 0)\leq v_L(p, 0)$, which implies \eqref{goal special Hamiltonian} by Theorem \ref{comparison thm}. 
\end{proof}

\section{Convexity preserving properties}\label{sec convex}
It is well known that {the} convexity preserving property holds for a large class of fully nonlinear equations in the Euclidean space; see \cite{GGIS}. Concerning convexity in the Heisenberg group, the notion of h-convexity (and equivalently v-convexity) turns out to be a natural extension of the Euclidean version. However, we cannot expect such convexity {to be preserved} in general. In fact, h-convexity is not preserved even for the first order linear equation. 

\begin{example}[Linear first order equations]\label{linear first example}
We again consider the linear equation \eqref{linear first eqn} with $h_0=(1, 1)$ and $u(x, y, z, 0)=f(x, y, z)$ with $f$ defined as in \eqref{h-convex fun} for all $(x, y, z)\in \mathbb{H}$. Let $h=(1, 1, 0)\in \mathbb{H}_0$. As verified in Example \ref{ex h-convex}, $u(\cdot, 0)$ is h-convex in $\mathbb{H}$. However, the unique solution 
\begin{equation}\label{linear first solution}
u(p, t)=f(p\cdot ht)=(x+t)^2(y+t)^2+2\left(z+{1\over 2}xt-{1\over 2}yt\right)^2
\end{equation}
is not h-convex for any $t>0$. In fact, the symmetrized Hessian {is given by} 
\[
(\nabla^2_H u)^\ast(p, t)=\begin{pmatrix} 2(y+t)^2+(y-t)^2 & 4(x+t)(y+t)-(x-t)(y-t) \\ 4(x+t)(y+t)-(x-t)(y-t) & 2(x+t)^2+(x-t)^2\end{pmatrix}.
\]
It is therefore easily seen that 
\[
(\nabla^2_H u)^\ast (t, t, 0, t)=\begin{pmatrix} 8t^2 & 16 t^2\\ 16t^2 & 8t^2\end{pmatrix}, 
\]
which shows that $u(\cdot , t)$ is not h-convex around the point $p=(t, t, 0)\in \mathbb{H}$ for any $t>0$.

\end{example}

The loss of convexity preserving is due to the non commutativity of the Heisenberg group product. Although the h-convexity of a function is preserved under left translations, it is not necessarily preserved under right translations, as indicated in Example \ref{linear first example}. We therefore consider right invariant h-convexity next.

\subsection{Right invariant h-convexity preserving}

\begin{defi}[Right invariant h-convexity]\label{defi left convex}
Let $\Omega$ be an open set in $\mathbb{H}$ and $u: \Omega\to \mathbb{R}$ be an upper semicontinuous function. The function $u$ is said to be right invariant horizontally convex or right h-convex in $\Omega$, if for every $p\in \mathbb{H}$ and $h\in \mathbb{H}_0$ such that $[h^{-1}\cdot p, h\cdot p]\subset \Omega$, we have 
\begin{equation}\label{left h-convex eqn}
u(h^{-1}\cdot p)+u(h\cdot p)\geq 2u(p).
\end{equation}
\end{defi}

\begin{comment}
In order to avoid analysis at space infinity, we consider the following Cauchy-Dirichlet problem:
\begin{numcases}{}
u_t+F(\nabla_H u, (\nabla_H^2 u)^\ast)= 0 &\quad \text{ in $\Omega \times (0, \infty)$},\label{cd eqn}\\
u(\cdot, 0)=u_0 &\quad\text{ in $\overline{\Omega}$}, \label{cd initial}\\
u(x, t)=g(x, t)  &\quad\text{ for $(x, t) \in \partial\Omega\times [0, \infty)$}, \label{cd boundary}
\end{numcases}
where $\Omega$ is a bounded domain in $\mathbb{H}$. 
\end{comment}

\begin{thm}[Right invariant h-convexity preserving]\label{thm left h-convexity}
Suppose that the assumptions (A1), (A2) {and (A3)} hold. Let $u\in C(\mathbb{H}\times [0, \infty))$ be the unique viscosity solution of \eqref{special eqn}--\eqref{special initial} satisfying the growth condition (G).  Assume in addition that $f$ is right invariant concave in $\mathbb{H}\times \mathbb{R}^2$, i.e., 
\begin{equation}\label{assumption operator}
f(h^{-1}\cdot p, \xi)+f(h\cdot p, \eta)\leq 2f\left(p, {1\over 2}(\xi+\eta)\right)
\end{equation}
for all $p\in \mathbb{H}$, $h\in \mathbb{H}_0$ and $\xi, \eta\in \mathbb{R}^2$. %and $X, Y\in \mathbf{S}^2$. 
If $u_0$ is right invariant h-convex in $\mathbb{H}$, then so is $u(\cdot, t)$ for all $t\geq 0$.
\end{thm}

\begin{comment}
\begin{rmk}
Combined with the ellipticity \eqref{ellipticity}, the concavity assumption \eqref{assumption operator} is equivalent to the following:
\[
F(\xi, X)+F(\eta, Y)\leq 2F\left({1\over 2}(\xi+\eta), Z\right)
\]
for all $\xi, \eta\in \mathbb{R}^2$ and $X, Y, Z\in \mathbf{S}^2$ satisfying $X+Y\geq 2Z$.
\end{rmk}
\end{comment}

\begin{proof}[Proof of Theorem \ref{thm left h-convexity}]
By definition, we aim to show that
\[
u(h^{-1}\cdot p, t)+u(h\cdot p, t)\geq 2u(p, t)
\]
for any $p\in \mathbb{H}, h\in \mathbb{H}_0, t\geq 0$. We assume by contradiction that there exist $(p_0, h_0, t_0)\in \mathbb{H}\times \mathbb{H}_0\times [0, \infty)$ such that 
\[
u(h_0^{-1}\cdot p_0, t_0)+u(h_0\cdot p_0, t_0)< 2u(p_0, t_0).
\]
Then there exists a positive maximizer $(\hat{p}, \hat{h}, \hat{t})\in \mathbb{H}\times \mathbb{H}_0\times [0, T)$ of 
\[
2u(p, t)-u(h^{-1}\cdot p, t)-u(h\cdot p ,t)-3 \sigma g(p, t)-{\sigma\over m-|h|_G^4}-{\sigma \over T-t}
\]
with some constants $m>|\hat{h}|_G^4$, $T>\hat{t}$ and $\sigma>0$ small. Here $g(p, t)=e^{\alpha t+\beta \la p\ra}$ with $\alpha>0$ to be determined later and any fixed $\beta>k$. We next consider 
\[
\begin{aligned}
\Phi(p, q, r, h, t, s, \tau)=2u(r, \tau)-u(h^{-1}\cdot p, t)& -u(h\cdot q, s)-\sigma \Psi_\vep(p, q, r, t, s, \tau)\\ -\psi_\vep(t, s, \tau)
& -{\sigma \over m-|h|_G^4}-{\sigma \over T-\tau},
\end{aligned}
\]
where
\[
\psi_\vep(t, s, \tau)={(t-s)^2\over \vep}+{(t-\tau)^2\over \vep}+{(s-\tau)^2\over \vep},
\]
\[
\Psi_\vep(p, q, r, t, s, \tau)=\phi_\vep(p, q, r)+ K(p, q, r, t, s, \tau)\]
with
\[
\phi_\vep(p, q, r)={|p\cdot r^{-1}|^4\over \vep}+{|q\cdot r^{-1}|^4\over \vep}
\]
and 
\[
K(p, q, r, t, s, \tau)=g(r, \tau)+g(p, t)+g(q, s).
\]
It follows that $\Phi$ has a maximizer $(p_\vep, q_\vep, r_\vep, h_\vep, t_\vep, s_\vep, \tau_\vep)$. As before, we denote 
\[
p_\vep=(x_{p_\vep}, y_{p_\vep}, z_{p_\vep}),\ q_\vep=(x_{q_\vep}, y_{q_\vep}, z_{q_\vep}),\ r_\vep=(x_{r_\vep}, y_{r_\vep}, z_{r_\vep}).
\]

 Due to the penalization at space infinity, we have $\phi_\vep(p_\vep, q_\vep, r_\vep)$ bounded from above uniformly in $\vep$. 
By a standard argument, we can show that there exists $\overline{p}\in \mathbb{H},\overline{h}\in \mathbb{H}_0$ and $\overline{t}\in [0, T)$ such that, up to a subsequence, 
\begin{equation}\label{convex limit1}
p_\vep, q_\vep, r_\vep\to \overline{p},\quad h_\vep\to \overline{h}, \quad t_\vep, s_\vep, \tau_\vep\to \overline{t}
\end{equation}
and 
\begin{equation}\label{convex limit2}
\phi_\vep(p_\vep, q_\vep, r_\vep)\to 0
\end{equation}
as $\vep\to 0$.

Since $u_0$ is right invariant h-convex, we have $\overline{t}>0$ and therefore $t_\vep, s_\vep, \tau_\vep>0$ when $\vep$ is sufficiently small. 

Denote $u_-(p, t)=u(h_\vep^{-1}\cdot p, t)$ and $u_+(p, t)=u(h_\vep\cdot p, t)$.
We now apply the Crandall-Ishii lemma in the Heisenberg group and get, for any $\lambda\in (0, 1)$,
\[
\begin{aligned}
(a_1, \zeta_1, X_1)\in & \overline{J}^{2, -} u_-(p_\vep, t_\vep),  \ (a_2, \zeta_2, X_2)\in \overline{J}^{2, -} u_+(q_\vep, s_\vep),\\
 & (a_3, \zeta_3, X_3)\in \overline{J}^{2, +} u(r_\vep, \tau_\vep)
\end{aligned}
\]
such that 
\begin{equation}\label{convex time der}
2a_3-a_1-a_2={\sigma \over (T-\tau_\vep)^2}+\sigma \alpha K(p_\vep, q_\vep, r_\vep, t_\vep, s_\vep, \tau_\vep),
\end{equation}
\smallskip
the horizontal projections of $\zeta_i$ can be expressed as $\xi_i+\eta_i$ ($i=1, 2, 3$) with
\[
-\xi_1=\sigma \nabla_H^p \phi_\vep(p_\vep, q_\vep, r_\vep),\ -\xi_2=\sigma \nabla_H^q \phi_\vep(p_\vep, q_\vep, r_\vep), \ 2\xi_3=\sigma \nabla_H^r \phi_\vep(p_\vep, q_\vep, r_\vep),
\]
\[
-\eta_1=\sigma \nabla_H g(p_\vep, t_\vep), \ -\eta_2=\sigma \nabla_H g(q_\vep, s_\vep), \ 2\eta_3=\sigma \nabla_H g(r_\vep, \tau_\vep),
\]
and
\smallskip
\begin{equation}\label{ishii2}
\langle (2X_3-X_1-X_2) w, w\rangle \leq \la (\sigma M+\lambda \sigma^2 M^2) w_{p_\vep}\oplus w_{q_\vep} \oplus w_{r_\vep}, w_{p_\vep}\oplus w_{q_\vep}\oplus w_{r_\vep}\ra,
\end{equation}

\noindent for all $w\in \mathbb{R}^2$, where $M=\nabla^2 \Psi_\vep(p_\vep, q_\vep, r_\vep, h_\vep, t_\vep, s_\vep, \tau_\vep)$ is a $9\times 9$ symmetric matrix, $w_{p_\vep}$, $w_{q_\vep}$ are respectively taken as in \eqref{wpe} and \eqref{wqe}, 
and 
\[
w_{r_\vep}=\left(w_1, w_2, {1\over 2}w_2 x_{r_\vep}-{1\over 2}w_1 y_{r_\vep}\right).
\]
By calculation, it is easily seen that 
\begin{equation}\label{gradient relation}
2\xi_3=\xi_1+\xi_2.
\end{equation}

We next set 
\[
M_1=\nabla^2 \phi_\vep(p_\vep, q_\vep, r_\vep)
\]
and
\[
M_2 =\nabla^2 K (p_\vep, q_\vep, r_\vep, t_\vep, s_\vep, \tau_\vep)
\]
Then $M=M_1+M_2$.  {To investigate the right hand side of \eqref{ishii2}, we first give estimates the terms involving $M_1$, which is a variant of \eqref{bieske} for three space variables. Note that 
\[
M_1=M_1' +M_1'',
\]
where 
\[
M_1'={1\over \vep}\nabla^2 {|p\cdot r^{-1}|^4}, \quad M_1''={1\over \vep}\nabla^2 {|q\cdot r^{-1}|^4}.\]
By direct calculations, we get
\[
\begin{aligned}
&\la M_1' w_{p_\vep}\oplus w_{q_\vep}\oplus w_{r_\vep}, w_{p_\vep}\oplus w_{q_\vep}\oplus w_{r_\vep}\ra=0,\\
&\la M_1'' w_{p_\vep}\oplus w_{q_\vep}\oplus w_{r_\vep}, w_{p_\vep}\oplus w_{q_\vep}\oplus w_{r_\vep}\ra=0
\end{aligned}
\]
and
\[
\begin{aligned}
&\la M_1'^2 w_{p_\vep}\oplus w_{q_\vep}\oplus w_{r_\vep}, w_{p_\vep}\oplus w_{q_\vep}\oplus w_{r_\vep}\ra={512 \over \vep} m_1^2|w|^2,\\
&\la M_1''^2 w_{p_\vep}\oplus w_{q_\vep}\oplus w_{r_\vep}, w_{p_\vep}\oplus w_{q_\vep}\oplus w_{r_\vep}\ra={512 \over \vep} m_2^2|w|^2,\\
&\la M_1'M_1'' w_{p_\vep}\oplus w_{q_\vep}\oplus w_{r_\vep}, w_{p_\vep}\oplus w_{q_\vep}\oplus w_{r_\vep}\ra={256 \over \vep} m_1m_2|w|^2,
\end{aligned}
\]
where 
\[
m_1=z_{p_\vep}-z_{r_\vep}+{1\over 2}y_{p_\vep} x_{r_\vep}-{1\over 2}x_{p_\vep} y_{r_\vep},\quad 
m_2=z_{q_\vep}-z_{r_\vep}+{1\over 2}y_{q_\vep} x_{r_\vep}-{1\over 2}x_{q_\vep} y_{r_\vep}.
\]
}
It follows that 
\begin{equation}\label{bieske convex0}
\la (M_1+\lambda M_1^2) w_{p_\vep}\oplus w_{q_\vep}\oplus w_{r_\vep}, w_{p_\vep}\oplus w_{q_\vep}\oplus w_{r_\vep}\ra\leq 
{512\lambda\over \vep} (m_1^2+m_2^2+ m_1m_2)|w|^2,
\end{equation}
which implies that
\begin{equation}\label{bieske convex}
\la (M_1+\lambda M_1^2) w_{p_\vep}\oplus w_{q_\vep}\oplus w_{r_\vep}, w_{p_\vep}\oplus w_{q_\vep}\oplus w_{r_\vep}\ra\leq {C\over \vep}|w|^2(|p_\vep\cdot r_\vep^{-1}|^4+|q_\vep\cdot r_\vep^{-1}|^4)
\end{equation}
for some $C>0$ independent of $\vep$ and $\lambda$. On the other hand, with the help of computer algebra system, we obtain estimates similar to \eqref{heisenberg second0}, \eqref{heisenberg second1} and \eqref{heisenberg second2}. In fact, we get a constant $C_\beta$ such that, when $\lambda>0$ is small enough (depending on $\vep$),\\
\begin{equation}\label{heisenberg second00}
\la  M_2 (w_{p_\vep}\oplus w_{q_\vep}\oplus w_{r_\vep}), (w_{p_\vep}\oplus w_{q_\vep}\oplus w_{r_\vep}) \ra\leq C_\beta|w|^2K(p_\vep, q_\vep, r_\vep, t_\vep, s_\vep, \tau_\vep)
\end{equation}

\begin{equation}\label{heisenberg second3}
\begin{aligned}
\la \lambda &(M_1M_2+M_2M_1+M_2^2) (w_{p_\vep}\oplus w_{q_\vep}\oplus w_{r_\vep}), (w_{p_\vep}\oplus w_{q_\vep}\oplus w_{r_\vep}) \ra\\
&\leq 2C_\beta|w|^2K(p_\vep, q_\vep, r_\vep, t_\vep, s_\vep, \tau_\vep)
\end{aligned}
\end{equation}

\noindent for some constant $\mu>0$ independent of $\vep, \beta$ and $\sigma$ and satisfying \eqref{penalty est}. As remarked in the proof of  Theorem \ref{comparison thm}, we obtain the constant $C_\beta$ thanks to the boundedness of $\nabla_H \la p\ra$ and $\nabla^2_H \la p \ra$ in $\mathbb{H}$.

Combining \eqref{ishii2}, \eqref{bieske convex} and \eqref{heisenberg second3}, we have 
\begin{equation}\label{est2 convex}
\begin{aligned}
\la (2X_3-X_1-X_2) & w, w\ra \leq {C\sigma \over \vep}\|w\|^2(|p_\vep\cdot r_\vep^{-1}|^4+|q_\vep\cdot r_\vep^{-1}|^4)\\
&+2\sigma \|w\|^2 C_\beta K(p_\vep, q_\vep, r_\vep, t_\vep, s_\vep, \tau_\vep)
.
\end{aligned}
\end{equation}
when $\lambda>0$ and $\sigma>0$ are sufficiently small. 

Since the horizontal derivatives are left translation invariant, the functions $u_-$ and $u_+$ are respectively solutions of 
\[
(u_-)_t-\tr (A \nabla^2_Hu_-)+ f(h_\vep^{-1}\cdot p, \nabla_H u_-)=0 \quad \text{ in $\mathbb{H}\times (0, \infty)$}
\]
and
\[
(u_+)_t-\tr (A \nabla^2_Hu_+)+ f(h_\vep\cdot p, \nabla_H u_+)=0 \quad \text{ in $\mathbb{H}\times (0, \infty)$}.
\]

 Applying the definition of viscosity subsolutions and supersolution, we have
\begin{equation}\label{solution1}
a_1-\tr (AX_1)+f(h_\vep^{-1}\cdot p_\vep, \xi_1+\eta_1)\geq 0,
\end{equation}
\begin{equation}\label{solution2}
a_2-\tr (AX_2)+f(h_\vep\cdot q_\vep, \xi_2+\eta_2)\geq 0,
\end{equation}
\begin{equation}\label{solution3}
a_3-\tr (AX_3)+f(r_\vep, \xi_3+\eta_3)\leq 0.
\end{equation}

Subtracting \eqref{solution1} and \eqref{solution2} from twice \eqref{solution3}, we get
\begin{equation}\label{est3 convex}
2a_3-a_1-a_2\leq \tr A(2X_3-X_1-X_2)+E,
\end{equation}
where 
\[
E=f(h_\vep^{-1}\cdot p_\vep, \xi_1+\eta_1)+f(h_\vep\cdot q_\vep, \xi_2+\eta_2)-2f(r_\vep, \xi_3+\eta_3)
\]
It follows from the concavity assumption \eqref{assumption operator}, the relation \eqref{gradient relation} and (A1)-(A2) that 
\begin{equation}\label{est4 convex}
\begin{aligned}
E\leq & \ f(h_\vep^{-1}\cdot p_\vep, \xi_1+\eta_1)-f(h_\vep^{-1}\cdot r_\vep, \xi_1+\eta_1)+f(h_\vep\cdot q_\vep, \xi_2+\eta_2)-f(h_\vep\cdot r_\vep, \xi_2+\eta_2)\\
&+f(h_\vep^{-1}\cdot r_\vep, \xi_1+\eta_1)+f(h_\vep\cdot r_\vep, \xi_2+\eta_2)-2f(r_\vep, {1\over 2}(\xi_1+\xi_2+\eta_1+\eta_2))\\
&+2f(r_\vep, {1\over 2}(\xi_1+\xi_2+\eta_1+\eta_2))-2f(r_\vep, \xi_3+\eta_3)\\
&\leq L_R(|h_\vep^{-1}\cdot r_\vep\cdot p_\vep^{-1}\cdot h|+|h_\vep^{-1}\cdot r_\vep\cdot q_\vep^{-1}\cdot h_\vep|)+L_f|\eta_1+\eta_2-2\eta_3|
\end{aligned}
\end{equation}
with $R=(|\overline{p}|+1)$ and $\vep>0$ small. Also, by \eqref{penalty est}, we have
\[
|\eta_1+\eta_2-2\eta_3|\leq 2(|\eta_1|+|\eta_2|+|\eta_3|)=2\sigma \beta\mu K(p_\vep, q_\vep, r_\vep, t_\vep, s_\vep, \tau_\vep).
\]

In view of \eqref{est2 convex}, \eqref{est3 convex} and \eqref{est4 convex}, we then obtain
\begin{equation}\label{convex inequality}
\begin{aligned}
& 2a_3-a_1-a_2\\
\leq & {C\sigma \over \vep}(|p_\vep\cdot r_\vep^{-1}|^4+|q_\vep\cdot r_\vep^{-1}|^4)+L_R|h_\vep^{-1}\cdot r_\vep\cdot p_\vep^{-1}\cdot h|+L_R|h_\vep^{-1}\cdot r_\vep\cdot q_\vep^{-1}\cdot h_\vep|\\
& +2\sigma  (C_\beta\|A\|+L_f\beta\mu)K(p_\vep, q_\vep, r_\vep, t_\vep, s_\vep, \tau_\vep)
.
\end{aligned}
\end{equation}
In view of \eqref{convex limit1} and \eqref{convex limit2}, we can take $\vep>0$ small such that 
\[
{C\sigma \over \vep}(|p_\vep\cdot r_\vep^{-1}|^4+|q_\vep\cdot r_\vep^{-1}|^4)+L_R|h_\vep^{-1}\cdot r_\vep\cdot p_\vep^{-1}\cdot h|+L_R|h_\vep^{-1}\cdot r_\vep\cdot q_\vep^{-1}\cdot h_\vep|<{\sigma \over T^2},
\]
which, by \eqref{convex inequality}, implies
\[
2a_3-a_1-a_2\leq {\sigma \over T^2}+2\sigma  (C_\beta\|A\|+L_f\beta\mu)K(p_\vep, q_\vep, r_\vep, t_\vep, s_\vep, \tau_\vep).
\]
It clearly contradicts \eqref{convex time der} when $\alpha$ is chosen to satisfy
\[
\alpha> 2\|A\|C_\beta+2L_f\beta\mu.
\]
\end{proof}
\begin{rmk}
The concavity assumption \eqref{assumption operator} on the operator $f$ is stronger than the assumptions of the convexity results in the Euclidean space as shown in \cite{GGIS, Ju}. In particular, the concavity of $\xi\mapsto f(p, \xi)$ is not needed in the Euclidean case. We here need this assumption, since there are no expressions of h-convexity in $\mathbb{H}$ corresponding to the following one for the Euclidean convexity
\[
u(\xi)+u(\eta)\geq 2u\left({\xi+\eta\over 2}\right)
\] 
for all $\xi, \eta\in \mathbb{R}^n$. It is not clear to us whether the assumption \eqref{assumption operator} can be weakened. 
\end{rmk}
\begin{example}
Let us revisit Example \ref{linear first example}. Since the equation \eqref{linear first eqn} and the solution \eqref{linear first solution}  satisfy all of the assumptions in Theorem \ref{thm left h-convexity}, the right invariant h-convexity of the solution is preserved, though the h-convexity is not. Indeed, if $u(p, t)$ is given by \eqref{linear first solution}, then by direct calculation we obtain, for all $p=(x, y, z)$, $h=(h_1, h_2, 0)$ and $t\geq 0$,
\[
\begin{aligned}
&u(h\cdot p, t)+u(h^{-1}\cdot p, t)\\
=&\  (x+h_1+t)^2(y+h_2+t)^2+2\left(z+{1\over 2}h_1y-{1\over 2}h_2x+{1\over 2}(x+h_1)t-{1\over 2}(y+h_2)t\right)^2\\
=&\ 2(x+t)^2(y+t)^2+4\left(z+{1\over 2}xt-{1\over 2}yt\right)^2+(h_1(y+t)-h_2(x+t))^2+2h_1^2(y+t)^2\\
&\ +2h_2^2(x+t)^2 +8(x+t)(y+t)h_1h_2+2h_1^2h_2^2\\
=&\ 2(x+t)^2(y+t)^2+4\left(z+{1\over 2}xt-{1\over 2}yt\right)^2+ 3(h_1(y+t)+h_2(y+t))^2+2h_1^2h_2^2\\
\geq &\ 2(x+t)^2(y+t)^2+4\left(z+{1\over 2}xt-{1\over 2}yt\right)^2=2u(p, t).
\end{aligned}
\]
\end{example}

\subsection{Left invariant h-convexity preserving}

We next discuss some special cases, where h-convexity and right invariant h-convexity are equivalent.

\begin{prop}[Evenness]\label{evenness lemma}
Let $u$ be an even or vertically even function on $\mathbb{H}$ . Then $u$ is h-convex in $\mathbb{H}$ if and only if $u$ is right invariant h-convex in $\mathbb{H}$. 
\end{prop}

\begin{proof}
By definition, $u$ is h-convex if $u$ satisfies \eqref{h-convex eqn} for any $p\in \mathbb{H}$ and $h\in \mathbb{H}_0$. Since $u$ is even, it is easily seen that \eqref{h-convex eqn} holds if and only if 
\[
u(h\cdot \overline{p})+u(h^{-1}\cdot \overline{p})\geq u(\overline{p}), 
\]
where $\overline{p}$ is given as in \eqref{reflection}, or 
\[
u(h\cdot p^{-1})+u(h^{-1}\cdot p^{-1})\geq u(p^{-1}) 
\]
for all $p\in \mathbb{H}$ and $h\in \mathbb{H}_0$, which is equivalent to saying 
\[
u(h\cdot p)+u(h^{-1}\cdot p)\geq u(p) \text{ for all $p\in \mathbb{H}$ and $h\in \mathbb{H}_0$}.
\]

\end{proof}

Another sufficient condition for equivalence between the h-convexity and the left h-convexity of a function $u$ on $\mathbb{H}$ is that $u$ has a separate structure; namely, 
\begin{equation}\label{separate}
u(x, y, z)=f(x, y)+g(z)
\end{equation}
for any $(x, y, z)\in \mathbb{H}$. 

\begin{prop}[Separability]\label{separability lemma}
Let $u$ be a function on $\mathbb{H}$ with a separate structure as in \eqref{separate}.  Then $u$ is h-convex in $\mathbb{H}$ if and only if $u$ is right invariant h-convex in $\mathbb{H}$. 
\end{prop}
\begin{proof}
Suppose $u$ can be written as in \eqref{separate}. Setting $p=(x, y, z)$ and $h=(h_1, h_2)$, we then have
\[
\begin{aligned}
&u(p\cdot h)=f(x+h_1, y+h_2)+g(z+{1\over 2}xh_2-{1\over 2}yh_1);\\
&u(p\cdot h^{-1})=f(x-h_1, y-h_2)+g(z-{1\over 2}xh_2+{1\over 2}yh_1);\\
&u(h\cdot p)=f(x+h_1, y+h_2)+g(z+{1\over 2}yh_1-{1\over 2}xh_2);\\
&u(h^{-1}\cdot p)=f(x-h_1, y-h_2)+g(z-{1\over 2}yh_2+{1\over 2}xh_2).
\end{aligned}
\]
It is easily seen that in this case
\[
u(p\cdot h^{-1})+u(p\cdot h)=u(h^{-1}\cdot p)+u(h\cdot p),
\]
which immediately yields the equivalence of \eqref{h-convex eqn} and \eqref{left h-convex eqn} in $\mathbb{H}$.
\end{proof}

The following result on preserving of the h-convexity itself is an immediate consequence of Theorem \ref{thm left h-convexity}, Propositions \ref{evenness lemma} and \ref{separability lemma}.
\begin{cor}[H-convexity preserving under evenness or separability]\label{cor convexity}
Assume that $f$ satisfies (A1)--{(A3)} and the concavity condition \eqref{assumption operator} for all $p\in \mathbb{H}$, $h\in \mathbb{H}_0$ and $\xi, \eta\in \mathbb{R}^2$. Let $u\in C(\mathbb{H}\times [0, \infty))$ be the unique viscosity solution of \eqref{special eqn}--\eqref{special initial} satisfying the growth condition (G).  Assume in addition that for any $t\geq 0$, $u(\cdot, t)$ either is an even or vertically even function or has a separable structure as in \eqref{separate}.  If $u_0$ is h-convex in $\mathbb{H}$, then so is $u(\cdot, t)$ in $\mathbb{H}$ for all $t\geq 0$.
\end{cor}

\subsection{More examples}\label{sec example}
In this section, we provide more examples, where the h-convexity is preserved. 

\begin{example}
Let $u_0(x, y, z)=(x^2+y^2)^2-8z^2$. It is not difficult to see that $u_0$ is an h-convex function in $\mathbb{H}$. Consider the heat equation 
\begin{equation}\label{heat eqn}
u_t-\Delta_H u=0 \quad \text{in $\mathbb{H}\times (0, \infty)$}
\end{equation}
with $u(\cdot , 0)=u_0$ in $\mathbb{H}$, where $\Delta_H$ denotes the horizontal Laplacian operator in the Heisenberg group, i.e., $\Delta_H u=\tr(\nabla_H^2 u)^\ast$. The unique solution of \eqref{heat eqn} in this case is 
\begin{equation}\label{heat sol1}
u(x, y, z, t)=(x^2+y^2)^2-8z^2+12(x^2+y^2)t+24t^2
\end{equation}
for all $(x, y, z)\in \mathbb{H}$ and $t\geq 0$ and it actually preserves the h-convexity of the initial value $u_0$. 
\end{example}

\medskip

\begin{example}
The solution as in \eqref{heat sol1} looks special, since it can be written as the sum of a function of $x, y, t$ and a function of $z$. A more complicated solution of the heat equation \eqref{heat eqn} is 
\begin{equation}\label{heat sol2}
u(x, y, z, t)=(x^2+y^2)z^2+{1\over 24}(x^2+y^2)^3+(4z^2+2(x^2+y^2)^2)t+17 (x^2+y^2) t^2+{68\over 3}t^3
\end{equation}
which contains mixed terms of $x, y$ and $z$. By direct calculation, one can also show that $u(\cdot, t)$ satisfies \eqref{v-convex eqn} in $\mathbb{H}$ in the classical sense for everywhere $t\geq 0$.  
\end{example}

\medskip

\begin{example}\label{example mcf}
We recall another example in \cite{FLM1} for the level-set mean curvature flow equation in $\mathbb{H}$.  The equation is of the form 
\begin{equation}
u_t-|\nabla_H u| \dive_H \left( \nabla_H u\over |\nabla_H u|\right)=0 \quad \text{ in $\mathbb{H}\times (0, \infty)$},
\end{equation}
where $\dive_H$ stands for the horizontal divergence operator in the Heisenberg group.
An explicit solution is 
\[
u(x, y, z, t)=(x^2+y^2)^2+16z^2+12(x^2+y^2)t+12t^2.
\]
This is also an example of h-convexity preserving but unfortunately is not covered by our current results.
\end{example}

%%%%%%%%%%%%%%%%%%%%%%%%%%%%%%%%%%%%%%

%%%%%%%%%%%%%%%%%%%%%%%%%%%%%%%%%%%%%
\bibliographystyle{abbrv}%
\bibliography{bib_liu}%

\end{document}